\documentclass[reqno]{amsart}

\usepackage{hyperref} 
\usepackage{graphicx, float}
\usepackage{amsmath, amsthm, amsfonts, amssymb, physics}
\usepackage{tikz}

\def\ex{{\rm ex}}

\newcommand{\floor}[1]{\left\lfloor #1 \right\rfloor}

\newcommand\degvec[1]{\langle #1 \rangle}

\newcommand{\N}{\mathbb{N}}

\newtheorem{theorem}{Theorem}[section]
\newtheorem{lemma}[theorem]{Lemma}

\theoremstyle{definition}

\newtheorem{definition}[theorem]{Definition}

\title{Improved Upper Bound on the Linear Turán Number of the Crown}

\date{\today}

\author[W. Fletcher]{Willem Fletcher}
\address{Department of Mathematics and Statistics, Carleton College, Northfield, MN, USA}
\email{willemrfletcher@gmail.com}

\begin{document}
    
    \begin{abstract}
        A linear $3$-graph is a set of vertices along with a set of edges, which are three element subsets of the vertices, such that any two edges intersect in at most one vertex. The crown, $C$, is a specific $3$-graph consisting of three pairwise disjoint edges, called jewels, along with a fourth edge intersecting all three jewels. For a linear $3$-graph, $F$, the linear Turán number, $ex(n,F)$, is the maximum number of edges in any linear $3$-graph that does not contain $F$ as a subgraph. Currently, the best known bounds on the linear Turán number of the crown are
        \[ 6 \floor{\frac{n - 3}{4}} \leq \ex(n, C) \leq 2n. \]
        In this paper, the upper bound is improved to $ex(n,C) < \frac{5n}{3}$.
    \end{abstract}
    
    \maketitle
    
    \section{Introduction}
    
    A \textit{$3$-graph}, $H$, is a non-empty set, $V(H)$, whose elements are called vertices, along with a set, $E(H)$, of $3$-element subsets of $V(H)$ called edges. A \textit{linear} $3$-graph is a $3$-graph where any two edges intersect in at most one vertex. This paper only concerns linear $3$-graphs, and for the remainder of it, $3$-graph will be used to mean linear $3$-graph.
    
    If $H$ and $F$ are $3$-graphs, then $H$ is \textit{$F$-free} if it has no subgraph isomorphic to $F$. For a $3$-graph, $F$, and $n \in \N$, the linear Turán number $\ex(n,F)$ is the greatest number of edges in any $F$-free 3-graph on $n$ vertices.
    
    A $3$-tree is a $3$-graph that can be constructed as follows: start with a single edge and add edges one at a time so that each new edge intersects exactly one other edge when it is added. This paper concerns a specific $3$-tree called the crown or $C$. It consists of three pairwise disjoint edges, called jewels, and one edge, called the base, intersecting all three jewels (See figure \ref{fig-crown}). The crown was named in \cite{CFGGWY} where it was proven that every $3$-graph with minimum vertex degree $\delta(H) \geq 4$ contains a crown.
    
    \begin{figure}[H]
        \centering
        \includegraphics[width=0.25\linewidth]{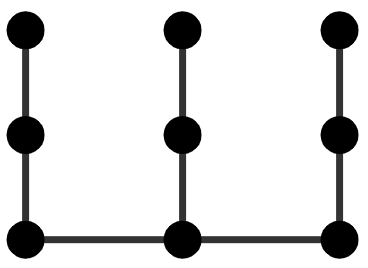}
        \caption{The crown}
        \label{fig-crown}
    \end{figure}
    
    In \cite{GYRS}, A. Gy\'arf\'as, M. Ruszink\'o, and  G. N. S\'ark\"ozy initiated the study of linear Tur\'an numbers for acyclic $3$-graphs, focusing on paths, matchings, and small $3$-trees. The crown is the smallest $3$-tree whose linear Turán number is not known. They provided the bounds
    
    $$ 6 \floor{\frac{n - 3}{4}} \leq \ex(n, C) \leq 2n. $$
    The purpose of this paper is to improve the upper bound.
    
    \begin{theorem} \label{main-theorem}
       For $n \geq 1$, $ex(n,C) < \frac{5n}{3}.$
    \end{theorem}
    
    The rest of this paper is dedicated to the proof of Theorem \ref{main-theorem}.
    
    \section{Preliminary Results and Notation}
    
    Let $H$ be a $3$-graph and $v \in V(H)$. The number of edges containing $v$ is the \textit{degree} of $v$ and is written $d(v)$. For an edge $e = \{a,b,c\} \in E(H)$, the $\textit{degree vector}$ of $e$ is $D(e) = \degvec{d(a),d(b),d(c)}$ where the coordinates are in non-increasing order. Define a partial order on the degree vectors by $D(e) \geq D(f)$ if $D(e)$ is greater than or equal to $D(f)$ in all three coordinates. The degree vector of an edge can be useful for finding a crown in a $3$-graph using the following result mentioned in \cite{CFGGWY} and \cite{GYRS}.
    
    \begin{lemma} \label{lemma-1}
        If $H$ is a $3$-graph and there is an edge $e \in E(H)$ such that $D(e) \geq \degvec{6,4,2}$, then $H$ is not crown-free.
    \end{lemma}
    
    \begin{proof}
    Let $e = \{a,b,c\} \in E(H)$ with $D(e)  = \degvec{d(a),d(b),d(c)} \geq \degvec{6,4,2}$. Then there is an edge $f \neq e$ containing $c$, an edge $g$ containing $b$ and not intersecting $f$, and an edge $h$ containing $a$ and not intersecting $f$ or $g$. Therefore, $H$ is not crown-free since $e,f,g, \text{ and } h$ form a crown.
    \end{proof}
    
    Define a \textit{counter-example} to be a crown-free $3$-graph, $H$, with $|E(H)| \geq \frac{5 |V(H)|}{3}$. To prove theorem \ref{main-theorem} we want to show that no such counter-example exists.
    
    \begin{lemma} \label{lemma-min-vertices}
        Any counter-example has a vertex set of at least 11 elements.
    \end{lemma}
    
    \begin{proof}
    Let $H$ be a counter-example with $|V(H)| = n$. By linearity, $|E(H)| \leq \frac{n(n-1)}{6}$ and since $H$ is a counter-example, $|E(H)| \geq \frac{5n}{3}$. This is only possible if $n \geq 11$. Thus, all counter-examples have at least $11$ vertices.
    \end{proof}
    
    A \textit{minimal counter-example} is a counter-example containing no proper subgraph that is also a counter-example. If $X \subseteq V(H)$ is a set of vertices, let $E_X(H)$ be the set of all edges containing at least one vertex in $X$.
    
    \begin{lemma} \label{lemma-inductive}
        If $H$ is a minimal counter-example, then there does not exist a proper subset $X \subsetneq V(H)$ of vertices such that $|E_X(H)| \leq \frac{5 |X|}{3}$.
    \end{lemma}
    
    \begin{proof}
        If such an $X$ exists, removing $E_X(H)$ and $X$ gives a proper subgraph that is a counter-example, contradicting $H$ being minimal.
    \end{proof}
    
    It follows from lemma \ref{lemma-inductive} that there are no vertices of degree zero or one in a minimal counter-example. The proof of the following lemma is given in section \ref{sec-proof-lemma-3}.
    
    \begin{lemma} \label{lemma-3}
        If $H$ is a minimal counter-example, then there is no edge $e \in E(H)$ with $D(e) \geq \degvec{5,5,5}$.
    \end{lemma}
    
    If $e \in E(H)$ and the coordinates of $D(e) = \degvec{x,y,z}$ are summed, we get the value $s(e) = x+y+z$. Define $L(H) \subseteq V(H)$ to be the set of ``large" vertices, of degree 9 or higher. Then, for an edge $e \in E(H)$ with $D(e) = \degvec{x,y,z}$, we define a modified $s(e)$ by:
    
    \begin{equation*}
        s^*(e) = \begin{cases}
        \min\{ s(e), 15 \} &\text{if $x \geq 9$}, \\ s(e) & \text{ otherwise.}
        \end{cases}
    \end{equation*}
    
    If we sum $s^*(e)$ over all of the edges in $H$, we get the value $$T^*(H) = \sum_{e \in E(H) }s^*(e).$$ 
    
    \begin{lemma} \label{lemma-2}
        If $H$ is a minimal counter-example on $n$ vertices, then $$\frac{T^*(H)}{|E(H)|} \geq \frac{25n+14 |L(H)|}{\frac{5n+2}{3}}.$$
    \end{lemma}
    
    Lemma \ref{lemma-2} is proven in section \ref{sec-proof-lemma-2}. We now give the proof of Theorem \ref{main-theorem}.
    
    \begin{proof}
    Suppose for contradiction that there exists a counter-example. Let $H$ be a (possibly non-proper) subrgraph that is a minimal counter-example. If $L(H)$ is not empty, then $\frac{T^*(H)}{|E(H)|} > 15$ by lemma \ref{lemma-2} so there exists an $e \in E(H)$ such that $s^*(e) \geq 16$. By definition of $s^*$ there is no vertex $v \in e \cap L(H)$ and by lemma \ref{lemma-inductive} there is no vertex in $e$ of degree one. However, then lemma \ref{lemma-1} tells us there is a crown in $H$, a contradiction.
    
    Now, suppose that $L(H)$ is empty. Then, by lemmas \ref{lemma-min-vertices} and \ref{lemma-2}, $\frac{T^*(H)}{|E(H)|} > 14$ so there exists an $e \in E(H)$ such that $s^*(e) \geq 15$. By assumption, there is no vertex $v \in e \cap L(H)$. There is also no vertex of degree one in $e$ and by lemma \ref{lemma-3}, $D(e) \neq \degvec{5,5,5}$. Then lemma $\ref{lemma-1}$ again gives a contradiction. Therefore, no counter-example exists, proving Theorem \ref{main-theorem}.
    \end{proof}

    \section{Proof of Lemma \ref{lemma-3}} \label{sec-proof-lemma-3}
    
    \begin{definition}[Link graph] \label{linkgraph}
         Let $H$ be a $3$-graph and $e = \{a, b, c\} \in E(H)$. The \textit{link graph} of $e$ in $H$ is the graph, $G(e)$, with a proper coloring, $\varphi$, of the edges defined as follows. The vertex set of $G(e)$ is all vertices covered by $E_e(H)$ minus $a$, $b$, and $c$. The edge set is $\{y,z\}$ for $\{x,y,z\} \in E_e \setminus \{e\}$ where $x \in e$. Such an edge, $\{y,z\}$, is assigned color $\varphi(\{y,z\}) = x$.
     \end{definition}
     
     A \textit{rainbow matching} in $G(e)$ is a set of three pairwise disjoint edges representing all three edge colors. There is a crown in $H$ with base $e$ if and only if there is a rainbow matching in $G(e)$.
     
     \begin{lemma} \label{triple-five-lemma}
        Let $H$ be a crown-free $3$-graph and $e = (a, b, c) \in E(H)$ with $D(e) = \degvec{5, 5, 5}$. Then $G(e)$ is isomorphic to the graph, $G$ (see Figure \ref{fig-link}).
    \end{lemma}
    
    
    \begin{figure}[H]
        \centering
        \includegraphics[width=.7\linewidth]{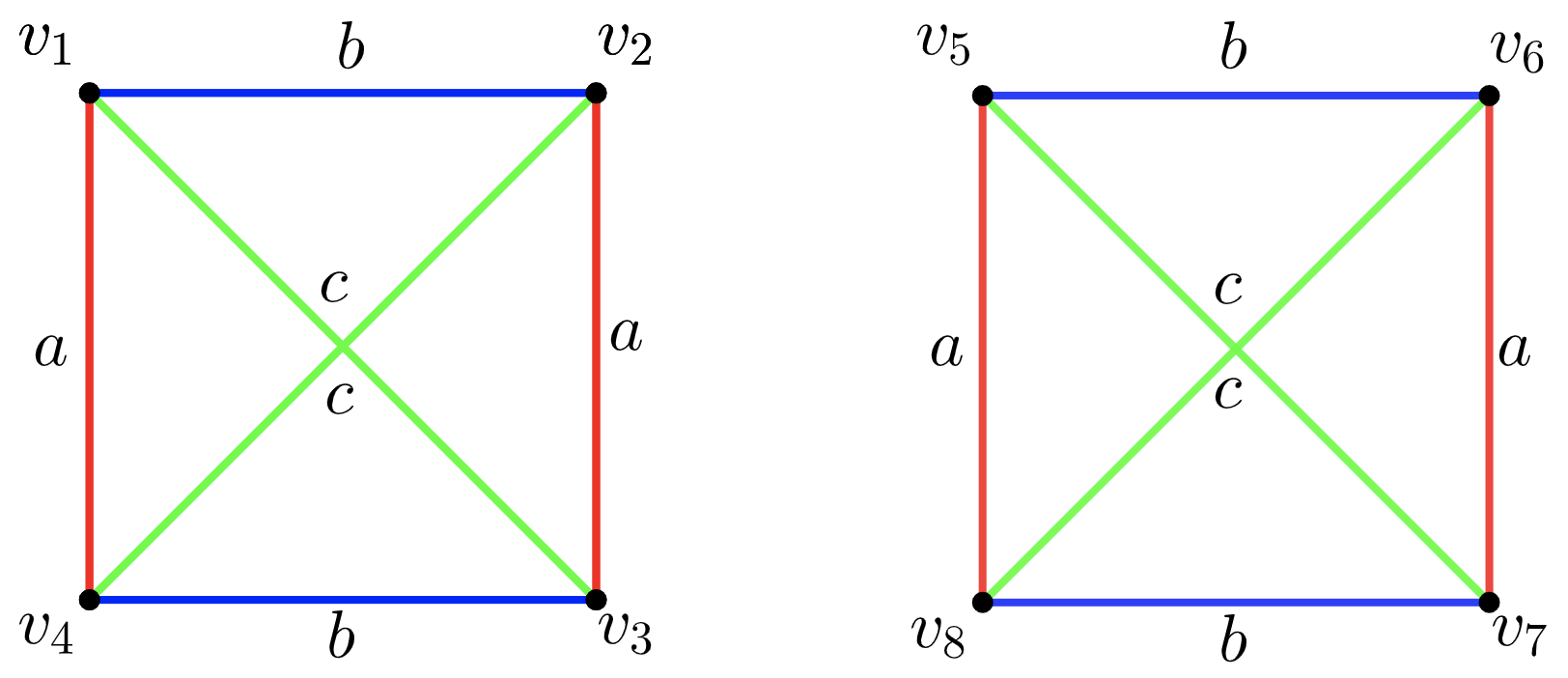}
        \caption{Link Graph G}
        \label{fig-link}
    \end{figure}
    
    \begin{proof}
    
    Given $i,j \in \{a, b, c \}$ where $i \neq j$, every edge of color $i$ must intersect two edges of color $j$, otherwise there is a rainbow matching. Therefore, the $a-$ and $b$-colored edges either form an alternating $8$-cycle or two disjoint alternating $4$-cycles. In the former case, any $c$-colored edge gives a rainbow matching. In the latter case, the only possible $c$-colored edges are the diagonals of the $4$-cycles, in which case $G(e)$ is isomorphic to $G$.
    \end{proof}

    Let $H$ be a minimal counter-example. Suppose for contradiction that there exists an edge $e = \{ a,b,c \} \in E(H)$ with $D(e) \geq \degvec{5,5,5}$. By lemma \ref{lemma-1}, $D(e) = \degvec{5,5,5}$ so $G(e)$ is isomorphic to $G$. Let $X = \{a, b, c \} \bigcup_{i=1}^8 \{v_i\}$ where $v_i$ are the vertices of $G$ (see figure \ref{fig-link}). If $E_X(H) = E_e(H)$, then
    
    $$ |E_X(H)| = |E_e(H)| = 13 < \frac{5 |X|}{3} $$
    contradicting $H$ being minimal by lemma \ref{lemma-inductive}.
    
    Thus, there exists an edge $f \in E_X(H) \setminus E_e(H)$. The edge $f$ may not contain any vertex from $e$, but must contain at least one $v_i$. By linearity and symmetry, we may assume $f = \{v_1, w_1, w_2 \}$ where $w_2 \not \in V\big(G(e)\big)$ and either $w_1 = v_5$ or $w_1 \not \in V\big(G(e)\big)$. In either case, there is a crown in $H$ with base 
    $$ \{v_1, v_4, a\} $$
    and jewels
    $$ \{v_1, w_1, w_2\} \text{, } \{v_2, v_4, c\} \text{, and } \{v_6, v_7, a\} \text{,} $$
    contradicting $H$ being a counter-example.\qed
    
    \section{Proof of Lemma \ref{lemma-2}}
    \label{sec-proof-lemma-2}
    
    Let $H$ be a minimal counter-example on $n$ vertices. By lemma \ref{lemma-inductive} there are no degree one vertices in $H$. Since $H$ is minimal, $$\sum_{v \in V(H)} d(v) = 5n + l$$ where $l \in \{ 0, 1, 2 \}$.
    
    \begin{lemma}
    For some $k > 0$ there exists a sequence of functions $\{f_i\}_{i=0}^k: V(H) \rightarrow \N$ and a partition of $V(H)$ into $I$ and $D = V(H) \setminus I$ such that:
    
    \begin{enumerate}
        \item For all $v \in V(H)$, $5 \leq f_0(v) \leq 7$.
        \item $f_k = d$ where $d$ is the degree function.
        \item For $ 1 \leq i \leq k$, there exists $x_i \in I$ and $y_i \in D$ such that $f_{i-1}(x_i) \geq f_{i-1}(y_i) $, $f_i(x_i) = f_{i-1}(x_i) + 1$, $f_i(y_i) = f_{i-1}(y_i) - 1$, and for all $z \in V(H) \setminus \{x_i, y_i\}$, $f_i(z) = f_{i-1}(z)$.
        \item If $v \in D$, $f_0(v) = 5$.
    \end{enumerate}
    
    \end{lemma}
    
    \begin{proof}
    Label the vertices in $V(H)$ as $v_1$ through $v_n$ so that $\{ d(v_i) \}_{i=1}^n$ is in non-decreasing order. First, suppose $l = 0$. Then, let $f_0(v_i) = 5$ for all $v \in V(H)$. To define $f_i$ for $i > 0$, assume $f_{i-1}$ is already defined. Take the minimal $a$ such that $f_{i-1}(v_a) > d(v_a)$ and the maximal $b$ such that $f_{i-1}(v_b) < d(v_b)$. If no such $a$ and $b$ exist, then $i-1 = k$ and we are done. Otherwise, define $f_i(v_a) = f_{i-1}(v_a) - 1$, $f_i(v_b) = f_{i-1}(v_b) + 1$, and $f_i(v_c) = f_{i-1}(v_c)$ for $c \not \in \{a,b\}$. Eventually, we reach an $f_i = d$. Any given vertex may be either decreased (one or more times) and is in the set $D$ or is never decreased and is in the set $I$. If $l = 1$, do the same except with $f_0(v_n) = 6$. If $l = 2$ and there is a vertex $v \in V(H)$ with $d(v) \geq 7$, let $f_0(v_n) = 7$. If no such vertex exists, then there are at least two vertices of degree six so let $f_0(v_n) = f_0(v_{n-1}) = 6$ and the final criteria must still be satisfied.
    \end{proof}
    
    Consider such a sequence of functions. We now define several things. For $0 \leq i \leq k$, let $$T_i = \sum_{v \in V(H)} f_i(v)^2.$$ For $1 \leq i \leq k$, let $\Delta_i = T_i - T_{i-1}$. For $v \in V(H)$, let $I_v = \{1 \leq i \leq k: f_i(v) \neq f_{i-1}(v) \}$. For $v \in V(H)$, define $$\Delta_v = \sum_{i \in I_v} \Delta_i.$$ For a fixed $i$, $1 \leq i \leq k$, there is one $v \in V(H)$ such that $f_i(v) > f_{i-1}(v)$. Let $g(i) = f_i(v)^2 - f_{i-1}(v)^2$. Similarly, define $h$ so that $g(i) - h(i) = \Delta_i$. Note that $$T_k = \sum_{v \in V(H)} d(v)^2 = \sum_{e \in E(H)} s(e)$$ and $\Delta_i > 0$ for all $1 \leq i \leq k$.
    
    \begin{lemma}
     If $v \in L(H)$ and $d(v) = m$, then $\Delta_v \geq m^2-9m+14$.
    \end{lemma}
    
    \begin{proof}
     We have $\Delta_v = \sum_{i \in I_v} \big(g(i)-h(i)\big)$. Now, $f_0(v) = x \in \{ 5,6,7 \}$ and $f_k(v) = m$ so $\sum_{i \in I_v} g(i) = m^2 - x^2$. For any $i \in I_v$, $h(i) \leq 5^2 - 4^2 = 9$. Thus, $\sum_{i \in I_v} h(i) \leq 9 |I_v| = 9 (m-x)$. Therefore, $\Delta_v \geq m^2 - x^2 - 9(m-x)$ which is minimal when $x=7$, giving $\Delta_v \geq m^2-9m+14$.
    \end{proof}
    
    \begin{lemma}
     If $v \in L(H)$ and $d(v) = m$, then $$\sum_{e \in E_v} \big( s(e) - s^*(e) \big ) \leq m^2 - 9m.$$
    \end{lemma}
    
    \begin{proof}
    Since $d(v) \geq 9$, the other two vertices in an edge with $v$ have maximum degree $3$ by lemma \ref{lemma-1}. Thus, if $e \in E_v$, $s(e) \leq m + 6$. Since there are $m$ edges in $E_v$, we have $\sum_{e \in E_v} \big ( s(e) - s^*(e) \big ) \leq m(m+6-15) = m^2 - 9m$.
    \end{proof}
    
    These two lemmas now give:
    
    \begin{align*}
        T^*(H) & = T_k - \sum_{v \in L(H)}\Big(\sum_{e \in E_v} \big (s(e) - s^*(e)\big)\Big) \\
        & \geq T_0 + \sum_{v \in L(H)} \Delta_v - \sum_{v \in L(H)}\Big(\sum_{e \in E_v} \big(s(e) - s^*(e)\big)\Big) \\
        & = T_0 + \sum_{v \in L(H)} \Big(\Delta_v - \sum_{e \in E_v} \big(s(e) - s^*(e)\big)\Big) \\
        & \geq T_0 + \sum_{v \in L(H)} \Big(\big(d(v)^2-9d(v)+14\big) - \big(d(v)^2 - 9d(v)\big)\Big) \\
        & = T_0 + \sum_{v \in L(H)} 14 \\
        & \geq 25n + 14 |L(H)|
    \end{align*}
    
    Since $|E(H)| = \frac{5n+l}{3} \leq \frac{5n+2}{3}$,
    
    $$ \frac{T^*(H)}{|E(H)|} \geq \frac{25n+14 |L(H)|}{\frac{5n+2}{3}}. $$ \qed
    
    \section{Acknowledgements}
    
    The work presented here was done in part at the Budapest Semesters in Mathematics Summer Undergraduate Research Program in 2021. Thanks to Andr\'as Gy\'arf\'as for introducing me to this problem and providing guidance.

\end{document}